\documentclass[reqno,12pt]{amsart}

\usepackage{enumerate}
\usepackage[margin=1in]{geometry}
\usepackage{ifpdf}
\usepackage{amsmath}
\usepackage{amsfonts}
\usepackage{amssymb}
\usepackage{amsthm}
\usepackage{amscd}
\usepackage[hyperfootnotes=false]{hyperref}
\usepackage{setspace}
\usepackage{amsrefs}
\usepackage{nicefrac}
\usepackage{makecell}
\usepackage{mathscinet}

\newcommand{\ignore}[1]{}

\renewcommand{\Re}{\operatorname{Re}}
\renewcommand{\Im}{\operatorname{Im}}
\newcommand{\Orb}{\operatorname{Orb}}

\newcommand{\CRdim}{\operatorname{CRdim}}

\newcommand{\abs}[1]{\left\lvert {#1} \right\rvert}

\newcommand{\C}{{\mathbb{C}}}
\newcommand{\R}{{\mathbb{R}}}

\newcommand{\bC}{{\mathbb{C}}}
\newcommand{\bR}{{\mathbb{R}}}

\newcommand{\sC}{{\mathcal{C}}}
\newcommand{\sF}{{\mathcal{F}}}

\newcommand{\sH}{{\mathcal{H}}}

\newcommand{\sO}{{\mathcal{O}}}

\newcommand{\sV}{{\mathcal{V}}}
\newcommand{\sX}{{\mathcal{X}}}

\newcommand{\rank}{\operatorname{rank}}

\DeclareMathOperator{\mult}{mult}

\newtheorem{thm}{Theorem}[section]

\newtheorem{prop}[thm]{Proposition}

\newtheorem{cor}[thm]{Corollary}

\newtheorem{lemma}[thm]{Lemma}

\theoremstyle{definition}
\newtheorem{defn}[thm]{Definition}
\newtheorem{example}[thm]{Example}

\theoremstyle{remark}
\newtheorem{remark}[thm]{Remark}

\thanks{The first author was in part supported by NSF grant DMS 0900885. The
fourth author was in part supported by a scholarship from the Vietnam Education Foundation. The fifth author was in part supported by NSF grant DMS 1200652.}
\author{Ji\v{r}\'{\i} Lebl}
\address{Department of Mathematics, University of Wisconsin,
Madison, WI 53706, USA}
\curraddr{Department of Mathematics, Oklahoma State University,
Stillwater, OK 74078, USA}
\email{lebl@math.okstate.edu}
\author{Andr\'e Minor}
\address{Department of Mathematics, University of California
at San Diego, La Jolla, CA 92093-0112, USA}
\email{aminor@math.ucsd.edu}
\author{Ravi Shroff}
\address{Centre for Mathematics and its Applications,
Mathematical Sciences Institute,
Australian National University, Canberra, ACT 0200, Australia}
\email{ravi.shroff@anu.edu.au}
\author{Duong Son}
\address{Department of Mathematics, University of California, Irvine, CA 92697-3875}
\email{snduong@math.uci.edu}
\author{Yuan Zhang}
\address{Department of Mathematical Sciences, Indiana University - Purdue University, Fort Wayne, IN 46805, USA}
\email{zhangyu@ipfw.edu}
\date{December 2, 2013}

\title{CR singular images of generic submanifolds under holomorphic maps}

\begin{document}


\begin{abstract}
The purpose of this paper is to organize some results on the local geometry of
CR singular real-analytic manifolds that are images of CR manifolds via
a CR map that is a diffeomorphism onto its image.
We find a necessary (sufficient in dimension 2)
condition for the diffeomorphism to extend to a finite holomorphic
map. The multiplicity of this map is a biholomorphic invariant
that is precisely the Moser invariant of the image when it is a Bishop surface with vanishing Bishop invariant.
In higher dimensions,
we study Levi-flat CR singular images and we
prove that the set of CR singular points must be large, and in the
case of codimension 2, necessarily Levi-flat or complex.
We also show that there exist real-analytic
CR functions on such images that satisfy
the tangential CR conditions at the singular points, yet fail to extend
to holomorphic functions in a neighborhood.  We provide
many examples to illustrate the phenomena that arise.
\end{abstract}

\maketitle


\section{Introduction} \label{section:intro}

Let $M$ be a smooth real submanifold in $\bC^n$, $n\geq 2$.  Given
$p\in M$, let $ T^{0,1}_pM$ denote the CR tangent space to $M$ at $p$, i.e.,
the subspace of antiholomorphic vectors in $\bC T_p\bC^n$ that are also tangent
to $M$.  $M$ is called a \emph{CR submanifold} when the function $\phi(p)=
\dim_{\bC}T^{0,1}_pM$ is constant.
In this paper, we focus
on submanifolds for which $\phi$ has jump discontinuities. We call such an $M$ a \emph{CR singular submanifold} and call those points where $\phi$ is discontinuous \emph{CR singular points} of $M$.  A CR singular submanifold is necessarily
of real codimension at least 2.
A two-dimensional CR singular submanifold in $\C^2$ already has a rich
structure and
the biholomorphic equivalence problem in this situation has been extensively studied
by many authors, for example \cites{Bishop65, MoserWebster83, Moser85,
HuangKrantz95, Gong04, HuangYin09}.
If $M$ is real-analytic,then a local real-analytic parametrization from $\bR^2$ onto $M$ gives rise to a holomorphic map
from $\bC^2$ to $\bC^2$ whose restriction to $\bR^2$ is a diffeomorphism onto $M$.
This observation motivates us to consider the general situation in which a CR
singular submanifold of $\bC^n$ is a diffeomorphic image of a generic
submanifold of $\bC^n$ of the same codimension via a CR map into $\bC^n$.

To be specific, let $N \subset \C^n$ be a generic real-analytic CR manifold.
Let $f \colon N \to \C^n$
be a real-analytic CR map such that $f$ is a diffeomorphism onto its image $M = f(N)$, which is
CR singular at some point $p\in M$, and suppose that $M$ is generic at
some point.  We call such an $M$ a \emph{CR singular image}.
Since $N$ is real-analytic, $f$
extends to a holomorphic map $F$ from a neighborhood of $N$ in $\C^n$ into a
neighborhood of $M$ in $\C^n$.  If the map $F$ does not have constant rank at
a point, the image of that point is a CR singular point of $M$
(see Lemma~\ref{lemma:singjacob}). This observation was also made in
\cite{ER}, where the images of CR submanifolds under finite
holomorphic maps were studied.
We consider the following questions regarding a CR singular image
 $M$:
\begin{enumerate}[(i)]
\item What can be said about the holomorphic
extension $F$ of the map $f$?  In particular, when is $F$ a finite map?
\item What is the structure of the
set of CR singular points of $M$?
\item Does every real-analytic CR function (appropriately defined) on $M$ extend to a
neighborhood of $M$ in $\bC^n$?
\end{enumerate}

In Section~\ref{section:twodim}, we first consider a real-analytic submafold $M$ of real dimension $n$ in $\bC^n$ with
one-dimensional complex tangent at a point.
Any such $M$ is an image of a totally real $N$
under a CR map that is a diffeomorphism onto $M$.
In special coordinates for $M$ we can write the map $f$
in a rather explicit form. We then provide a necessary and sufficient condition for $f$ to extend to a \emph{finite} holomorphic map $F$,
in terms of the defining equation of $M$.
In this case, the multiplicity of $F$ is a biholomorphic invariant. In two dimensions,
this invariant is closely related to the Moser invariant
(see \cite{Moser85}) of $M$ when it is a Bishop surface with elliptic complex tangent,
or the lowest order
invariants studied in \cite{Harris85} by Harris.
Furthermore, by invoking a theorem of Moser \cite{Moser85}, we can show that
the only Bishop surface in $\C^2$ that cannot
be realized as an image of $\R^2 \subset \C^2$ via a finite holomorphic
map is the surface $M_0:=\{ w = \abs{z}^2 \}$ (Theorem~\ref{thm:n2complete}).
When $\dim_{\bR}M >n$, we are able to provide an explicit example of a CR singular manifold $M$
which is not an image of a generic submanifold
of the same codimension (see Example~\ref{ex:notanimage}).  When $M$ is a singular image, we
give a necessary condition  for the extended map to be finite, and show that the multiplicity
remains a biholomorphic invariant (see Section~\ref{section:prelim}).

We then study a CR singular image $M$ that contains complex subvarieties of positive dimension. The main result 
in Section~\ref{section:discfam} (Theorem~\ref{thm:generalsingthm}) shows
that if $M$ is a CR singular image with a CR singular set $S$, and
$M$ contains a complex subvariety $L$ of complex dimension $j$ that intersects $S$, then
$S \cap L$ is a complex subvariety of (complex) dimension $j$ or $j-1$.  Furthermore, if $M$ contains a continuous family of complex varieties $L_t$ of
dimension $j$ and $L_0 \cap S$ is of dimension $j-1$, then
$L_t \cap S$ is nonempty for all $t$ near 0.

We apply these ideas in
Theorem~\ref{thm:lfsing} to characterize the CR singular set of a Levi-flat
CR singular image. We construct several
examples illustrating the different possibilities for the 
structure
of the set of CR singular points.  A corollary to our theorem
shows that a Levi-flat CR singular image
necessarily has a CR singular set of large dimension, depending on
the generic CR dimension of the singular image.
 When $M$ is of real codimension is 2 in $\bC^n$, we obtain that the CR singular set
is necessarily Levi-flat or complex.  One of the primary motivations for studying CR
singular Levi-flat manifolds is to understand the singularities of
non-smooth Levi-flat varieties in general.  For example, it has been proved by the
first author~\cite{Lebl:lfsing}
that the singular locus of a singular Levi-flat hypersurface is
Levi-flat or complex.  The next step in
understanding the singular set of a Levi-flat hypersurface
would be to find
a Levi-flat stratification, for which we need to understand the
CR singular set that may arise in higher codimension flat submanifolds.

We next attempt to find a convenient set of coordinates for a nowhere
minimal CR singular image $M$ along the lines of the standard result, Theorem~\ref{nicenormcoord},
for CR manifolds.  In particular, a generic submanifold $N$
has coordinates in which some of the equations are of the
form $\Im w' = 0$, where $\{ w' = s \} \cap N$ give the CR orbits of $N$.
The theorem does not generalize directly, but when $M$ is a CR singular image under a finite holomorphic map, we obtain a partial analogue, namely that $M$ is
contained in the intersection of singular Levi-flat hypersurfaces.

In Section~\ref{section:noext}, we consider the extension of real-analytic CR functions defined on a CR singular image $M$.  When $M$ is generic at every point, then all real-analytic
CR functions on $M$ extend to holomorphic functions on a neighborhood of $M$
in $\bC^n$. In contrast, we show that if $M$ is a CR singular image, then there exists
a real-analytic function satisfying all tangential CR conditions, yet \emph{fails} to extend to a holomorphic function on a neighborhood of~$M$. This result is closely related to an earlier result \cite{H} in which the author provided
a necessary and sufficient condition for a CR function on the \emph{generic part}
of $M$ to extend across a CR singular point.

The authors would like to acknowledge Peter Ebenfelt for many
conversations on this and related subjects and for advice and
guidance over the years.


\section{Preliminaries} \label{section:prelim}

In this section we will recall some basic notations and results that will be needed
in the rest of the paper. We refer to \cite{BER:book} for more details.
Let $M$ be a smooth real submanifold of real codimension $d$ in $\bC^n$.
Then at every point $p\in M$, there is a neighborhood $U$ of $p$ and $d$ real-valued
smooth functions $r_1,\dots, r_d$ defined on $U$ such that
\begin{equation}
 M\cap U = \{z\in \bC^n: r_k(z,\bar{z}) = 0, \ k=1,2,\dots, d\},
\end{equation}
where $dr_1\wedge dr_2\wedge \dots \wedge dr_d$ does not vanish on $U$.
For any point $p\in M$, we denote by $T^{0,1}_pM$ the subspace of $T^{0,1}_p\bC^n$ that annihilates
$r_k$ for $k=1,2,\dots, d$. Thus, we see that
\begin{equation}
 \dim_{\bC}T^{0,1}_pM = n- \rank_{\bC}\left(\frac{\partial r_j}{\partial\bar{z}_k}(p,\bar{p})\right)_{j,k}.
\end{equation}
If $\dim_{\bC}T^{0,1}_qM$ is constant for all $q$ near $p$ and equals $n-d$,
then we say that $M$ is \emph{generic} at $p$. In this paper, we shall assume
that $M$ is connected,
real-analytic
and generic at some point and thus the matrix
$\left(\frac{\partial r_j}{\partial\bar{z}_k}\right)_{j,k}$ is of generic full rank.
If further we denote by $S$  the set of CR singular points of $M$, then
\begin{equation}
 S = \left\{z \in M :  \rank_{\bC}\left(\frac{\partial r_j}{\partial\bar{z}_k}\right)_{j,k} \le d-1 \right\}.
\end{equation}
The set of CR singular points $S$ is a \emph{proper} subvariety of $M$ and
$M\setminus S$ is generic at all points.

\begin{remark}
We note that when $M \subset \C^n$
is a CR singular submanifold such that there exists
a subbundle $\sV \subset \C \otimes TM$ such that $\sV_q = T^{0,1}_qM
\subset T^{0,1}_q \C^n$
for all $q \in M \setminus S$, then $(M,\sV)$ becomes an abstract
real-analytic CR manifold.  Hence, $(M,\sV)$ is locally integrable (see
\cite{BER:book}*{Theorem 2.1.11}).  Therefore, for every $p \in M$ we obtain a
generic $N \subset \C^n$ and a real-analytic CR map $f \colon N \to \C^n$
such that $f$ is a diffeomorphism onto an open neighbourhood of $p$ in $M$.
We shall call such pair $(N,f)$ (or simply $N$) a resolution of CR singularity of $M$ near $p$. The converse
is also true; if there exists a resolution of CR singularity of $M$ near $p$, then
the CR bundle on $M\setminus S$ extends to a subbundle of $\bC\otimes_{\bR} TM$ in a neighborhood of $p$ on $M$.
The resolution of CR singularity $N$,
if it exists,
is unique, modulo a biholomorphic equivalence (Proposition~\ref{prop:23}).
\end{remark}

Let $N \subset \C^n$ be a real-analytic
generic submanifold and $f \colon N \to \bC^n$
a real-analytic CR map that is a diffeomorphism onto its image $M=f(N)$.
By the real-analyticity, the map $f$ extends to a holomorphic map
$F$ in a neighborhood of $N$ in $\bC^n$.  One of the questions we are interested
in
is whether $F$ is a finite holomorphic map. For $p\in \bC^n$ we denote by $\sO_p$
the ring of germs of holomorphic functions at $p$.

\begin{defn}
A germ of a holomorphic map $F=(F_1, \dots, F_n)$ defined in a neighborhood of $p\in \C^n$
is said to be finite at $p$ if the ideal $\mathcal{I}(F)$ generated by
$F_1-F_1(p),
\dots, F_n - F_n(p)$ in
$\sO_p$ is of finite codimension, that is, if $\dim_{\bC}\sO_p/\mathcal{I}(F)$ is finite.
This number, denoted by $\mult_p(F)$, is called the multiplicity of $F$ at $p$.
\end{defn}

Equivalently, a holomorphic map 
 as above is finite if and only if the
germ at $p$ of
the complex analytic variety $F^{-1}\bigl(F(p)\bigr)$ is an isolated point
(cf. \cite{AGV85}).
In this case, for any $q$ close enough to $F(p)$, the number of preimages
$\# F^{-1}(q)$ is finite and always less than or equal to $\mult_p(F)$.
The equality holds for generic points in a neighborhood of $F(p)$.

%

\begin{prop}\label{prop:23}
Let $M, \widetilde{M} \subset \C^n$
be connected CR singular real-analytic submanifolds
that are generic at some point and $\varphi$ a biholomorphic map of a
neighbourhood of $M$ to a neighbourhood of $\widetilde{M}$ such that
$\varphi(M) = \widetilde{M}$.
Let $N, \widetilde{N} \subset \C^n$ be generic
real-analytic submanifolds,
$F$ be a holomorphic map from
a neighborhood of $N$ to a neighborhood of $M$, and
$\widetilde{F}$ be a holomorphic map from
a neighborhood of $\widetilde{N}$ to a neighborhood of $\widetilde{M}$,
such that $F|_N$ and
$\widetilde{F}|_{\widetilde{N}}$ and
are diffeomorphisms onto $M$ and $\widetilde{M}$ respectively.
Then $N$ and $\widetilde{N}$ are biholomorphically equivalent.

Furthermore, for any point $p \in M$, $\mult_p(F)$ is a
local biholomorphic invariant of $M$ (i.e., does not
depend on $N$ and $F$).
\end{prop}

\begin{proof}
Write $f = F|_N$ and $\widetilde{f} = \widetilde{F}|_{\widetilde{N}}$.
We have the
following commutative diagram 
\begin{equation}
\begin{CD}
N @>f>> M\\
@V {\widetilde{f}^{-1} \circ \varphi \circ f} VV       @VV {\varphi|_M} V\\
\widetilde{N}     @>\widetilde{f}>> \widetilde{M}.\\
\end{CD}
\end{equation}
Let $S$ be the set of CR singular points of $M$, and hence $\varphi(S)$
the set of CR singular points of $\widetilde{M}$.
As $\widetilde{M}$ is generic outside $\varphi(S)$ and the Jacobian
of $\widetilde{F}$ does not vanish on $\widetilde{M} \setminus \varphi(S)$, then
$\widetilde{f}^{-1}$ is a CR map on $\widetilde{M} \setminus \varphi(S)$.
The map $\widetilde{f}^{-1} \circ \varphi \circ f$ is a diffeomorphism that
is a CR map outside of $f^{-1}(S)$, which is nowhere dense in $N$.  It is, therefore,
a real-analytic CR diffeomorphism of the generic submanifolds $N$ and
$\widetilde{N}$, and so it extends to a biholomorphism of a neighbourhood.
By uniqueness of the extension of CR maps from generic submanifolds,
the diagram still commutes after we extend.  Hence, the extensions $F$
and $\widetilde{F}$ have the same multiplicity.
\end{proof}

If $M$ is a real-analytic submanifold of codimension $d=2$ in $\bC^n$, then $\dim_{\C} T^{0,1}_pM = n-1$ at a CR singular point $p\in M$. Thus we can find a linear change of coordinates such that the new coordinates
$Z=(z,w)\in \C^{n-1} \times \C$ vanishes at $p$ and $M$ is given by one complex
equation:
\begin{equation}
w = \rho(z,\bar{z}) ,
\end{equation}
where $\rho$ vanishes to order at least $2$ at $0$.

\begin{prop}
Suppose $M \subset \C^n$
is a CR singular real-analytic submanifold of codimension 2
near zero, defined by $w=\rho(z,\bar{z})$, where $\rho$ vanishes to
order at least 2 at $0$.  Suppose that $N$ is a real-analytic
generic submanifold of $\C^n$ and  $F$ a holomorphic map of a neighborhood of $N$ into $\C^n$ such that $F|_N\colon N \to \C^n$ is a diffeomorphism onto $M$.
If $\rho(0,\bar{z}) \equiv 0$, then $F$ cannot be a finite map.
\end{prop}

\begin{proof} Write $F=(F', F_n)$.  Let $N$
be given
in normal coordinates by (see \cite{BER:book}*{Proposition 4.2.12})
\begin{equation}
\omega = r(\zeta,\bar{\zeta},\bar{\omega}) ,
\end{equation}
where $(\zeta,\omega) \in \C^{n-2} \times \C^2$, and $r$ is
a $\C^2$-valued holomorphic function satisfying
$r(\zeta,0,\bar{\omega}) =
r(0,\bar{\zeta},\bar{\omega})
=
\bar{\omega}$.
Then for $(\zeta,\omega) \in N$ we have
\begin{equation}
F_n(\zeta,\omega) =
\rho\bigl(F'(\zeta,\omega),\bar{F'}(\bar{\zeta},\bar{\omega})\bigr) .
\end{equation}
By plugging in the defining equation of $N$ as
$\bar{\omega} = \bar{r}(\bar{\zeta},\zeta,\omega)$, we get that
\begin{equation} \label{eq:normcoordplugin}
F_n(\zeta,\omega) =
\rho\bigl(F'(\zeta,\omega),\bar{F'}(\bar{\zeta},\bar{r}(\bar{\zeta},\zeta,\omega))\bigr)
\end{equation}
holds near 0.  In particular \eqref{eq:normcoordplugin}
holds when $\bar{\zeta} = 0$.  Using
the fact that in normal coordinates $\bar{r}(0,\bar{\zeta},\omega) \equiv \omega$
we get
\begin{equation}
F_n(\zeta,\omega) =
\rho\bigl(F'(\zeta,\omega),\bar{F'}(0,\bar{r}(0,\zeta,\omega))\bigr)
=
\rho\bigl(F'(\zeta,\omega),\bar{F'}(0,\omega)\bigr) .
\end{equation}
So $F$ cannot be finite if $\rho(0,\bar{z}) \equiv 0$, as in that
case $F_n$ is in the ideal generated by components of $F'$.
\end{proof}

\begin{remark} \label{remark:complexcont}
We conclude this section by a remark that the existence of a finite
holomorphic map $F$ and a generic submanifold $N$ such that $F$ restricted to
$N$ is a diffeomorphism onto $M=F(N)$ implies that the set of CR singular points on
$M$ is contained in
a proper complex subvariety of $\C^n$.  This fact follows from
Lemma~\ref{lemma:singjacob}, which shows that the inverse image of the CR
singular points is contained in the set where the Jacobian of $F$ vanishes,
and from Remmert proper map theorem.
\end{remark}


\section{Images under finite maps of totally real submanifolds} \label{section:twodim}

Let $M$ be a real-analytic submanifold of real dimension $n$ in $\bC^n$.
Assume that $M$ is totally real at a generic point and
$S\subset M$ is the CR singular set. Suppose that $p\in S$ and $\dim_{\bC} T_p^{0,1}M = 1$. Then, after a change
of coordinates, we can assume
that $p=0$ and $M$ is defined by the following equations
\begin{equation} \label{eq:nmfld-define}
\begin{split}
z_n &= \rho(z_1,\bar z_1, x') ,\\
y_\alpha &= r_\alpha(z_1,\bar z_1, x'),\quad \alpha = 2,\dots, n-1 .
\end{split}
\end{equation}
Here, $z = (z_1,\dots, z_n)$ are the coordinates in $\bC^n$,
$z'=(z_2,\dots,
z_{n-1})$ (when $n=2$, $z'$ is omitted) and $z_j = x_j + i y_j$. The
functions $\rho$ and $r_\alpha$ have no linear terms,
and $r_\alpha$ are real-valued. By making another change of coordinates,
we can eliminate the harmonic terms in $r_\alpha$ to obtain
\begin{equation}\label{sonz}
 r_\alpha(0, \bar{z}, 0)=0.
\end{equation}

The case when $p$ is a ``nondegenerate'' CR singularity of $M$ 
 has been studied from a different view point
(see, e.g., \cites{Bishop65, Huang98, HuangYin12, KenigWebster84}). Here we make no such assumption.

\begin{prop}\label{totallyreal}
Let $M\subset \bC^n$ be a real-analytic submanifold of real dimension $n$, with a complex tangent at $0$,
defined by equations \eqref{eq:nmfld-define}.
Then the following are equivalent
\begin{enumerate}[(i)]
\item \label{thm:nmfld:i} There is a totally real submanifold $N$ of dimension $n$ in $\C^n$ and
a germ of a \emph{finite} holomorphic map $F \colon (\C^n,0) \to (\C^n,0)$ such
that $F|_N$ is diffeomorphic onto $M = F(N)$ (as germs at 0).
\item \label{thm:nmfld:ii} $\rho(0,\bar{z}_1, 0) \not\equiv 0$.
\end{enumerate}
Furthermore, if \eqref{thm:nmfld:ii} holds, then
\begin{equation}\label{eq:nmfld-rho1}
 \rho(0,\bar{z}_1, 0 ) = c\bar{z}_1^k + O(\bar{z}_1^{k+1}), \quad c\neq 0,
\end{equation}
with $k=\mult_0(F)$.
\end{prop}

\begin{proof}
We can assume that $N = \R^n \subset \C^n$.
Consider the map $f = (f_1, f', f_n)
\colon \R^n \to \C^n$ given by
\begin{align}\label{eq:nmfld-a1}
f_1(t_1, t', t_n) &= t_1+ i t_n,\\ \label{eq:nmfld-a2}
f_\alpha(t_1, t', t_n) &= t_\alpha + i r_\alpha(t_1+ i t_n, t_1- i t_n, t'),\\ \label{eq:nmfld-a3}
f_n(t_1, t', t_n) &= \rho(t_1+ i t_n, t_1- i t_n, t').
\end{align}
Then clearly, $f$ is a diffeomorphism from a neighborhood of $0$ in $\bR^n$
onto $M$. Furthermore, $f$ extends to a holomorphic map $F$ from a neighborhood of $0$
in $\bC^n$.
Thus, by abuse of notation, we also denote the coordinates in $\bC^n$ by $t$.
Let $V_F$ be the germ of $F^{-1}(0)$ at $0$.

\eqref{thm:nmfld:i} $\Rightarrow$ \eqref{thm:nmfld:ii}:
Assume that $\rho(0,\bar{z}, 0) \equiv 0$. By making use of the fact that $r_\alpha(0, \bar{z}, 0) \equiv 0$,
 one can check that
\begin{equation}
\{(t_1,t',t_n) \in \C^n: t_1 = -it_n, t'=0\} \subset V_F.
\end{equation}
Thus $V_F$ has positive dimension, and so $F$ is not finite at $0$.

\eqref{thm:nmfld:ii} $\Rightarrow$ \eqref{thm:nmfld:i}:
Assume that \eqref{thm:nmfld:ii} holds. Let $(t_1, t', t_n) \in V_F$.
From \eqref{eq:nmfld-a1} we have $t_1 = -i\, t_n$.
Substitute into \eqref{eq:nmfld-a2} and \eqref{eq:nmfld-a3} we get
\begin{align}\label{eq:nmfld-a4}
t_\alpha + i\, r_\alpha(0, 2t_1, t')& = 0\\ \label{eq:nmfld-a5}
\rho(0, 2t_1, t') &= 0.
\end{align}
As $r_\alpha$ has no linear terms, by the implicit function theorem,
we see that \eqref{eq:nmfld-a4} has a unique solution. Furthermore, since
$r_\alpha$ has no harmonic terms, the unique
solution must be $t' = 0$.

Substituting $t'=0$ into \eqref{eq:nmfld-a5}, we get
\begin{equation}
\rho\bigl(0, 2t_1, 0\bigr) = 0.
\end{equation}
Using \eqref{eq:nmfld-rho1} we get
\begin{equation}\label{caigiday}
 2^kc\, t_1^ k + O(t^{k+1}) = 0.
\end{equation}
We then deduce that $t_1=t_n =0$. Hence, $V_F = \{0\}$ is isolated and thus $F$ is finite.

Finally, let $\tilde{z}=(\tilde z_1, \tilde z', \tilde z_n')$ be a point close enough to $0$.
We will show that for generic $q$, in a small neighborhood of $0$ there are $k$ solutions
to the equation $F(t) = \tilde{z}$. Indeed, if $F(t) = \tilde{z}$ then
\begin{align}
 t_1+ i t_n &= \tilde z_1,\label{sona}\\
 t_\alpha + i\, r_\alpha(t_1+ i t_n, t_1 -  i t_n, t') &= \tilde z_\alpha,\label{sonb}\\
 \rho(t_1+ i t_n, t_1 -  i t_n, t') &= \tilde{z}_n. \label{sonc}
\end{align}
From \eqref{sona} we have
\begin{equation}\label{sond}
 t_1 = -i\,t_n + \tilde{z}_1.
\end{equation}
Substitute \eqref{sond} into \eqref{sonb} we get
\begin{equation}\label{sone}
 t_\alpha +i\, r_\alpha(\tilde{z}_1, 2t_1-\tilde{z}_1, t') = \tilde{z}_\alpha.
\end{equation}
For $\tilde{z}$ close enough to $0$, the implicit function theorem
applied to \eqref{sone} gives unique solution $t' = \varphi(t_1, \tilde{z}_1)$.
Substitute this into \eqref{sonc} we get
\begin{equation}\label{sonf}
 \rho(\tilde{z}_1, 2t_1-\tilde{z}_1, \varphi(t_1, \tilde{z}_1)) = \tilde{z}_n.
\end{equation}
From \eqref{sonz} we have that $\varphi(t_1,0) \equiv 0$. Thus, \eqref{sonf}
is a small pertubation of \eqref{caigiday} depending on the size of $\abs{\tilde{z}}$.
Thus, there is a neighborhood $U$ of $0$ such that for generic $\tilde{z}$ close
enough to $0$, the equation \eqref{sonf} has exactly $k$ solutions in $U$ for $t_1$,
by Rouch\'{e}'s theorem. Therefore,
$F^{-1}(\tilde{z})$ consists of $k$ distinct points near $0$,
and hence $\mult_0(F) =k$.
\end{proof}

In two dimensions, we
 get further results along this line of reasoning.
Let $M$ be a real-analytic surface (i.e., a $2$-dimensional real-analytic submanifold)
in $\C^2$ and $p\in M$. If $p$ is a CR singular point of $M$, then we can
find a change of
coordinates such that $p=0$ and $M$ is given by $w= \rho(z,\bar{z})$,
where $\rho$ vanishes to order at least $2$ at $0$.
A \emph{Bishop surface} is a surface where $\rho$ vanishes exactly
to order 2 at the origin.

As in Theorem~\ref{totallyreal}, we see that the condition $\rho(0,\bar z)\not\equiv 0$, says precisely when $M$
is the image of $\R^2$ under a finite map.
The following lemma is proved by noting that Segre varieties are
invariant under formal transformations.  We omit the details.

\begin{lemma}\label{son1}
Let $M$ and $M'$ be a CR singular real-analytic surfaces in $\C^2$
near $0$ given by
$w=\rho(z,\bar{z})$ and $w' = \rho'(z',\bar{z}')$ and
$F$ a formal invertible transformation that sends $M$ into $M'$. Then
$\rho(0,z)\equiv 0$ if and only if $\rho'(0,\bar z') \equiv 0$.
\end{lemma}
%

We conclude this section by the following theorem that completely
analyzes the situation in $n=2$.

\begin{thm} \label{thm:n2complete}
Let $(M,0)$ be a germ of a CR singular real-analytic surface at $0$ in $\C^2$.
Assume that $M$ is defined near $0$ by $w=\rho(z,\bar z)$, where
$\rho$ vanishes to
order at least 2 at $0$. Then the following are equivalent:
\begin{enumerate}[(i)]
\item \label{prop:n2:i} There is a germ of a totally real, real-analytic
surface $(N,0)$ in $\C^2$ and a \emph{finite}
holomorphic map $F \colon (\C^2,0) \to (\C^2,0)$ such that $F|_N$ is
diffeomorphic onto $M = F(N)$ (as germs at $0$).
\item \label{prop:n2:ii} $\rho(0,\bar z) \not\equiv 0$.
\end{enumerate}

In addition, if $M$ is a Bishop surface then
\eqref{prop:n2:i} and \eqref{prop:n2:ii} are equivalent to the following: $M$ is \emph{not} (locally) biholomorphically
equivalent to $M_0:=\{w = \abs{z}^2\}$. If the Bishop invariant $\gamma \neq 0$ then $\mult_0(F) = 2$. Otherwise,
$s = \mult_0(F)$ is the Moser invariant of $M$.
\end{thm}

\begin{proof}
The equivalence of \eqref{prop:n2:i} and \eqref{prop:n2:ii} follows from Theorem~\ref{totallyreal}.

Now suppose that $M$ is a Bishop surface given by $w=\rho(z,\bar z)$.
After a biholomorphic change of coordinates, we can write
\begin{equation}\label{bishop}
\rho(z,\bar{z}) =\abs{z}^2 + \gamma (z^2 + \bar{z}^2) + O(\abs{z}^3),
\end{equation}
where $0\leq\gamma\leq \infty$ is the Bishop invariant \cite{Bishop65}. When
$\gamma =\infty$ the equation \eqref{bishop} is understood as $\rho(z,\bar z)= z^2+\bar{z}^2 + O(\abs{z}^3)$.

Suppose that \eqref{prop:n2:i} or \eqref{prop:n2:ii} holds, then from Lemma~\ref{son1} we see that
$M$ is not equivalent to $M_0:=\{w=\abs{z}^2\}$. Conversely, assume that $M$ is not equivalent to $M_0$. Then either
$\gamma \neq 0$ and hence \eqref{prop:n2:ii} holds, or $\gamma=0$.
In the latter case, by the work of Moser
\cite{Moser85}, after a formal change of coordinates, $M$ can be brought
to a ``pseudo-normal'' form
\begin{equation}\label{moserwebster}
w = \abs{z}^2 + z^s + \bar z^s + \sum_{i+j\geq s+1} a_{ij}z^i \bar{z}^j.
\end{equation}
Here, $0\leq s \leq \infty$ is the Moser invariant.  Assume for contradiction that \eqref{prop:n2:ii}
does not hold. Then $\rho(0,\bar{z}) \equiv 0$. Notice that the property
$\rho(0,\bar{z}) \equiv 0$ is preserved under formal transformations carried
in \cite{Moser85}, by Lemma~\ref{son1}, we deduce that $s=\infty$ and so $M$ is formally
equivalent to $M_0$. By \cite{Moser85}, $M$ is biholomorphically equivalent
to $M_0$. We obtain a contradiction. 

Finally, if $\gamma \neq 0$ then from \eqref{bishop} we see that $\rho(0,\bar{z}) = \gamma \bar{z}^2 + O(\bar{z}^3)$ and
hence $\mult_0(F) = 2$. Otherwise, from \eqref{moserwebster}, we have $\rho(0,\bar z) = \bar{z}^s + O(\bar{z}^{s+1})$ and thus
$\mult_0(F) = s$ is the Moser invariant of $M$. 
\end{proof}


\section{Images containing a family of discs} \label{section:discfam}

The following theorem is one of the main tools to study
CR singular submanifolds containing complex subvarieties developed in this paper.

\begin{thm} \label{thm:generalsingthm}
Let $N \subset \C^n$ be a connected generic real-analytic
submanifold and let $f \colon N \to \C^n$ be a real-analytic CR map
that is a diffeomorphism onto its image, $M = f(N)$.
Let $S \subset M$ be the CR singular set of $M$ and suppose that
$M$ is generic at some point.
\begin{enumerate}[(i)]
\item \label{generalsingthm:scaplitem}
If $L \subset M$ is a complex subvariety of dimension $j$,
then $S \cap L$ is either empty or a complex
subvariety of $L$ of dimension $j-1$ or $j$.
\item \label{generalsingthm:generalscapl}
Let $A \colon [0,\epsilon) \times \overline{\Delta} \to M$ be
a family of analytic discs such that $A(0,0) = p \in S$.
If $A(0,\overline{\Delta}) \setminus S$ is nonempty, then
there exists an $\epsilon' > 0$ such that
$A(t,\overline{\Delta}) \cap S \not= \emptyset$ for all $0 \leq t < \epsilon'$.
\item \label{generalsingthm:jdimitem}
If $M$ contains a continuous one real dimensional family of complex
manifolds of complex dimension $j$, and if $S$ intersects one of the
manifolds, then $S$ must
be of real dimension at least $2j-1$.
\end{enumerate}
\end{thm}

By a one-dimensional family of
analytic discs we mean a map
\begin{equation}
A \colon I \times \overline{\Delta} \to \C^n ,
\end{equation}
where $I \subset \R$ is an interval, $\Delta \subset \C$ is the unit disc,
$A$ is a continuous function such
that $z \mapsto A(t,z)$ is nonconstant and holomorphic in $\Delta$ for every
$t \in I$.

In essence, part \eqref{generalsingthm:scaplitem} says that if $M$ contains
a complex variety, the intersection of this variety with $S$ must be large
(and complex).
Part \eqref{generalsingthm:generalscapl} of the theorem says
that if there exists a one-dimensional family of complex varieties in $M$,
and $S$ intersects one of them properly, then it must intersect all of them.
Part \eqref{generalsingthm:jdimitem} puts the two parts together.

In the proof of Theorem~\ref{thm:generalsingthm}, we need the following
lemma.

\begin{lemma} \label{lemma:singjacob}
Let $N \subset \C^n$ be a
connected generic submanifold and
let $F \colon \C^n \to \C^n$ be a holomorphic map
such that $f = F|_N$ is
a diffeomorphism onto its image, $M = f(N)$.
If $d$ is the real codimension of $N$, then for $p \in N$ we have
\begin{equation}
\dim_{\C} T_p^{1,0} M
= 2n-d-\rank_{\C} \left[ \frac{\partial F_i}{\partial z_j} (p) \right] .
\end{equation}
In particular,
let $S \subset M$ be the CR singular set of $M$, and
suppose that $M$ is a generic submanifold at some point.  Then
\begin{equation}\label{mozart}
f^{-1}(S) = \{ z \in \C^n : J_F (z) = 0 \} \cap N .
\end{equation}
\end{lemma}

Here $J_F(z) = \det \left[ \frac{\partial F_i}{\partial z_j} (z) \right]$
denotes the holomorphic Jacobian of $F$.

\begin{proof}
Let $p\in N$ and $q=F(p) \in M$. We denote by $J$ the complex structure in $\C^n$ as usual. Since $N$ is generic,
\begin{equation}
T_pN + JT_pN = T_p\C^n.
\end{equation}
On the other hand, since $F$ is holomorphic, we have $F_*\circ J = J'\circ
F_*$.  Furthermore, $F_*(T_pN) = T_qM$ because $F|_N$ is diffeomorphism.
Therefore,
\begin{equation}
F_*(T_p\C^n) = F_*(T_pN + JT_pN) = F_*(T_pN) + F_*(JT_pN) = T_qM + JT_qM.
\end{equation}
Consequently,
\begin{equation}
\begin{split}
\rank_\R F_*|_p
&= \dim_\R F_*(T_p\C^n) \\
&= \dim_\R (T_qM + JT_qM) \\
&= \dim_\R T_qM + \dim_\R JT_qM - \dim_\R (T_qM \cap JT_qM).
\end{split}
\end{equation}
Since $F$ is holomorphic, the real rank $\rank_\R F_*|_p $ equals to
$2\rank_\C \left[ \frac{\partial F_i}{\partial z_j} (p) \right]$
(twice the rank of the complex Jacobian matrix of $F$). Hence
\begin{equation}
2\rank_\C \left[ \frac{\partial F_i}{\partial z_j} (p) \right] = 2(2n-d) - 2\dim_\C T_q^{1,0}M.
\end{equation}
In other words,
\begin{equation}
\dim_\C T_q^{0,1}M =
2n-d-\rank_\C \left[ \frac{\partial F_i}{\partial z_j} (p) \right].
\end{equation}

The second part of the lemma now follows at once.
\end{proof}

\begin{lemma}\label{lemma:nonextendible}
Let $N \subset \C^n$ be a connected generic real-analytic submanifold and
let $f \colon N \to \C^n$ be a CR map
that is a real-analytic diffeomorphism onto its image, $M = f(N)$. Let $S$ be the
CR singular set of $M$ and suppose that $M$ is generic at some point.
If $p\in S$,
then there is a neighborhood $U$ of $p$ in $M$ and a real-analytic function $u$ on $U$
that is CR on $U\setminus S$
that does not extend to a holomorphic function past~$p$.
\end{lemma}

\begin{proof}
Let $q=f^{-1}(p) \in N$ and let $F$ be the unique holomorphic extension of $f$ to some neighborhood
of $N$ in $\C^n$. Then by Lemma~\ref{lemma:singjacob}, $F$ has degenerate rank
at $q$. Assume for contradiction that for all neighborhoods $U$ of $p$, every real-analytic function $g$ on
$U$ that is CR on $U\setminus S$ extends to a holomorphic function past $p$. We claim that the homomorphism
\begin{equation}
F^*\colon \mathcal{O}_p \to \mathcal{O}_q, \quad  F^*(g) = g\circ F
\end{equation}
is surjective, where $\mathcal{O}_p$ and $\mathcal{O}_q$ are the rings of
germs holomorphic functions at $p$ and $q$. For if $h$ is a holomorphic function
in a neighborhood $V$ of
$q$ in $\C^n_t$, we consider the function $u = \left(h|_N\right) \circ f^{-1}$. Clearly, $u$ is a real-analytic function
on $U$ that is CR on $U \setminus S$, where $U=f(V\cap N)$. By assumption
$u$ extends past $p$
to an element $\hat{u} \in \mathcal{O}_p$.  It is straightforward to verify that
$F^*(\hat{u}) = h$ on $M\setminus S$ near $p$. By the genericity of $M$ at points in $M\setminus S$, we obtain that $F^*(\hat{u}) = h$ as germs near $p$ and hence, the claim follows.
In particular, there are germs of holomorphic functions $g_j$ at $p$, such that the coordinate functions
$t_j = g_j\circ F$. Let $G=(g_1,\dots, g_j)$, then $G$ is a germ of a
holomorphic map
satisfying $G \circ F = \mathrm{Id}$. This is impossible since $F$ has degenerate
rank at $q$. 
\end{proof}

We also need the following result of
Diederich and Forn\ae ss~\cite{DF:realbnd}*{claim in section 6}.

\begin{lemma}[Diederich-Forn\ae ss] \label{DF:lemma}
Let $U \subset \C^n$ be an open set and
let $S \subset U$ be a real-analytic subvariety.  For every
$p \in S$, there exists a neighborhood $U'$ of $p$ such that
for every $q \in U'$ and every germ of a complex variety
$(V,q) \subset (S,q)$, there exists a (closed) complex subvariety $W \subset
U'$ such that $(V,q) \subset (W,q)$ and such that $W \subset S \cap U'$.
\end{lemma}

The lemma has the following useful corollary.

\begin{cor} \label{cor:localcplxisglobalcplx}
Let $U \subset \C^n$ be an open set and
let $X \subset U$ be a real-analytic subvariety.
Suppose that there exists an open dense set $E \subset X$
such that
for every $p \in E$ there exists a neighborhood $U'$ of
$p$ such that $X \cap U'$ is a complex manifold.  Then $X$ is a
complex analytic subvariety of $U$.
\end{cor}

\begin{proof}[Proof of Theorem~\ref{thm:generalsingthm}]
Let us begin with \eqref{generalsingthm:scaplitem}.  First look at the
inverse image $f^{-1}(L)$.  This set is a real subvariety of $N$, though
we cannot immediately conclude that $f^{-1}(L)$ is a complex variety.

Let $F$ be the unique holomorphic extension of $f$ to a neighborhood of $N$.
Near points of $N \setminus f^{-1}(S)$, the map $F$ is locally biholomorphic,
by Lemma~\ref{lemma:singjacob}.
Hence, $f^{-1}(L) \setminus f^{-1}(S)$ is a complex analytic variety.

If $f^{-1}(L) \setminus f^{-1}(S)$ is empty, then we are finished as
$L \subset S$.  Let us assume that $L$ is irreducible.  As $f$ is a
diffeomorphism, $f^{-1}(L)$ is also an irreducible subvariety.
So suppose that
$f^{-1}(L) \setminus f^{-1}(S)$ is nonempty, and therefore
an open dense subset of $f^{-1}(L)$.
By applying Corollary~\ref{cor:localcplxisglobalcplx}
we obtain that $f^{-1}(L)$ must be a complex variety.  As $F^{-1}(S)$ is defined
by a single holomorphic function $J_F$ and furthermore, it follows from \eqref{mozart} that
\begin{equation}
f^{-1}(L\cap S) = f^{-1}(L)\cap f^{-1}(S) = f^{-1}(L)\cap F^{-1}(S),
\end{equation}
then $E:= f^{-1}(L \cap S)$ must
be a complex subvariety of dimension $j-1$ (or empty).  As $L\cap S$ is a
real-analytic subvariety, we can invoke
Corollary~\ref{cor:localcplxisglobalcplx} again to conclude that $L\cap S$
is complex variety as follows.
If $p\in L\cap S$ is a regular point and $q=f^{-1}(p)\in E$ then $q$
is a regular point of $E$ as $f$ is a diffeomorphism.  Furthermore, $E$ is a
complex manifold near $q$, i.e., $T_q E=T_q^c E$.
Here, $T_q^c E:=T_q E \cap J(T_qE)$
is the complex tangent space at $q$ of $E$. Thus, as $f$ is a CR map and diffeomorphism,
\begin{equation}
 T_p(L\cap S) = f_*(T_q E) = f_*(T_q^c E) \subset T_p^c (L\cap S).
\end{equation}
Consequently, $T_p(L\cap S) = T_p^c(L\cap S)$.
Therefore, $L\cap S$ is a complex manifold near $p$. Since the
regular part of $L \cap S$ is dense, $L\cap S$ is a complex manifold
near all points on an open dense subset of $L\cap S$ and hence $L\cap S$ is
complex variety.

Let us now move to \eqref{generalsingthm:generalscapl}.
As $S \cap A(0,\Delta)$
is a nonempty proper subset, we know it is a proper complex
subvariety by \eqref{generalsingthm:scaplitem}.  Without
loss of generality we can rescale $A$ such that
$S \cap A(0,\Delta) = \{ p \}$.  Suppose for contradiction that
$S \cap A(t,\overline{\Delta})$ is empty for all $0<t\leq \epsilon'$, in
other words, $A(t,\overline{\Delta}) \subset M\setminus S$ for all $0<t\leq \epsilon'$ for some $\epsilon'>0$.  Then
by the Kontinuit\"atssatz (see, e.g., \cite{Shabat:book}*{page 190}) any holomorphic function
defined on a neighborhood of $M \setminus S$ extends to a neighborhood of
$A(0, \overline \Delta)$ (and hence past $p$).
In view of Lemma~\ref{lemma:nonextendible}, we obtain a contradiction. 

\eqref{generalsingthm:jdimitem} follows as a consequence of
\eqref{generalsingthm:scaplitem} and \eqref{generalsingthm:generalscapl}.
\end{proof}


\section{Levi-flats} \label{section:lf}

In this section we study CR singular,
Levi-flat submanifolds in $\C^n$.  Unlike in two dimensions where the CR dimension is 0, a codimension
2, CR submanifold in $\C^3$ or larger
must have nontrivial CR geometry.  The simplest
case is the Levi-flat.  A CR manifold is said to be \emph{Levi-flat} if the
Levi-form vanishes identically. Equivalently, for each $p\in M$, there is
a neighborhood $U$ of $p$ such that $M\cap U$ is foliated by complex manifolds
whose leaves $L_c$ satisfy $T_qL_c = T_qM \cap J(T_qM)$ for all $q\in M \cap V$ and all $c$.
This foliation is unique, it is simply the foliation by CR orbits.

In the real-analytic case, a generic Levi-flat submanifold of codimension $d$ is
locally biholomorphic to the submanifold defined by
\begin{equation} \label{eq:deflf}
\Im z_1 = 0, \Im z_2 = 0 , \dots, \Im z_d = 0 ,
\end{equation}
that is, a submanifold locally equivalent to $\R^d \times \C^{n-d}$.
The situation is different if we allow a CR singularity. The fact that
there are infinitely many different CR singular
Levi-flat submanifolds not locally biholomorphically equivalent is already
evident from the theory of Bishop surfaces in $\C^2$ (see Section~\ref{section:twodim}).

We apply the result in the previous section to study CR singular Levi-flats
that are images of a neighborhood of $\R^d \times \C^{n-d}$ under a CR
diffeomorphism.  We show that in dimension 3 and higher, unlike in 2
dimensions, there exist Levi-flats that are not images of a CR submanifold.
First, let us study the CR singular set of an image of $N = \R^d \times
\C^{n-d}$. In this case, the Levi foliation on $N$  gives rise to a
real-analytic foliation $\mathcal L$ on $M$ that coincides with the
Levi-foliation on $M\setminus S$. Moreover, it is readily seen that the leaves of
$\mathcal L$ are complex submanifolds (even near points in $S$) and thus $M$
is also foliated by complex manifolds. We call a ``leaf'' of $M$ a leaf of
this foliation.

\begin{thm} \label{thm:lfsing}
Let $n \geq 3$, $n > d \geq 2$.
Let $U \subset \R^d \times \C^{n-d}$ be a connected open set and let
$f \colon U \to \C^n$ be a real-analytic CR map that is a diffeomorphism onto
its image, $M = f(U)$.
Let $S \subset M$ be the CR singular set of $M$ and suppose that
$M$ is generic at some point.
\begin{enumerate}[(i)]
\item \label{lfsing:item1}
If $(x,\xi)$ are the coordinates in $\R^d \times \C^{n-d}$, then
the set $f^{-1}(S)$ is locally the zero set of a real-analytic
function that is holomorphic in $\xi$.
\item \label{lfsing:item2}
If $L$ is a leaf of $M$, then $S \cap L$ is either
empty or a complex
analytic variety of dimension $n-d$ or $n-d-1$ .
\item \label{lfsing:item3}
If $S \cap L$ is of dimension $n-d-1$, then
$S$ must intersect all the leaves in some neighborhood of $L$.
\end{enumerate}
\end{thm}

In particular, the theorem says that the CR singularity cannot be isolated
if $M$ is an image of $\R^{d} \times \C^{n-d}$.  In fact, the CR singularity
cannot be a real one-dimensional curve either.  It cannot be a curve inside
a leaf $L$ as it is a complex variety when intersected with $L$.  When
it is a point, $S$ must intersect all leaves nearby, and there is at
least a 2-dimensional family of leaves of $M$.  That is, the singular set
is always 2 or more dimensional.

For example when $n=3$, $d=2$, it is possible that the singular set is either
2 or 3 dimensional.  We show below that where it is CR it must be
Levi-flat in the following sense.  If we include complex manifolds
among Levi-flat manifolds we can say that
a CR submanifold $K$ (not
necessarily a generic submanifold) is \emph{Levi-flat} if near every $p \in K$
there exist local coordinates $z$ such that $K$ is defined by
\begin{align}
& \Im z_1 = \Im z_2 = \dots = \Im z_j = 0 \\
& z_{j+1} = z_{j+2} = \dots = z_{j+k} = 0.
\end{align}
for some $j$ and $k$, where we interpret $j=0$ and $k=0$ appropriately.
This definition includes complex manifolds ($j=0$) and generic
Levi-flats ($k=0$), although we generally call complex submanifolds
complex rather than Levi-flat.

\begin{cor} \label{cor:lfcor}
Let $n \geq 3$, $n > d \geq 2$.
Let $U \subset \R^d \times \C^{n-d}$ be a connected open set and let
$f \colon U \to \C^n$ be a real-analytic CR map that is a diffeomorphism onto
its image, $M = f(U)$.
Let $S \subset M$ be the CR singular set of $M$ and suppose that
$M$ is generic at some point.
If $S$ is nonempty, then it is of real dimension at least $2(n-d)$.
Furthermore, near points where $f^{-1}(S)$ is a CR submanifold, it is
Levi-flat or complex.
\end{cor}

\begin{proof} Suppose $p\in S$ and $L$ is the leaf passing through $p$. From
Theorem~\ref{thm:lfsing}, $L\cap S$ is a complex variety of complex dimension
$n-d$ or $n-d-1$ near $p$. If $L\cap S$ is of dimension $(n-d)$ near $p$, then we are done.
Otherwise, using Theorem~\ref{thm:lfsing} again, we have that $S$ intersects a
$d$-parameter family of leaves near $p$. Therefore, the real dimension of $S$ near $p$ is $2(n-d-1) +d \geq 2(n-d)$.

Let $(x,\xi)$ be our parameters in $U$ as in the theorem, then
$f^{-1}(S)$ is given by a real-analytic function that is
holomorphic in $\xi$.  In other words, $f^{-1}(S)$ is a subvariety that is a Levi-flat
submanifold at all regular points where it is CR.  To see this fact, it is enough
to look at a generic point of $f^{-1}(S)$.
\end{proof}

The fact that $f^{-1}(S)$ is Levi-flat does not imply
that $S$ must be Levi-flat.  This is precisely because the complex Jacobian of the holomorphic
extension $F$ of $f$ vanishes on $f^{-1}(S)$.  In fact, as the examples in
Section~\ref{section:examples} suggest, the CR structure of $f^{-1}(S)$ and $S$
may be quite different.  In particular, the CR dimension of $S$
may be strictly greater than the CR dimension of $f^{-1}(S)$.
What we can say for sure is that through each point,
$S$ must contain complex varieties of complex dimension at least
$n-d-1$. Using Theorem~\ref{thm:generalsingthm}, we can obtain the following
information on the CR structure of $S$.

\begin{cor}
Let $M$ be as above and let $p$ be a generic point of $S$ (in particular, $S$ is CR near $p$). Suppose $L$ is the leaf on $M$ through $p$, that is, an image of the leaf of the
Levi-foliation through $f^{-1}(p)$ and let $E=S\cap L$.
\begin{enumerate}[(i)]
\item The following table lists all the possibilities for the CR structure of $S$ near $p$  when the codimension is $d=2$ (DNO means `Does not occur').
\medskip

\begin{center}
\begin{tabular}{|c|c|c|}
\hline
\diaghead{dimRSCRdimS}{$\dim_{\R}(S)$}{$\CRdim(S)$}
 & $n-3$ & $n-2$ \\
\hline
$2n-4$ &
\begin{tabular}{c} $S$ is Levi-flat \\$\dim_{\mathbb{C}}E= n-3$ \end{tabular}
&
\begin{tabular}{c} $S$ is complex \\
$\dim_{\mathbb{C}}E = n-2$ or \\$\dim_{\mathbb{C}}E = n-3$ \end{tabular}
\\ \hline
$2n-3$ & DNO & \begin{tabular}{c}  $S$ is Levi-flat \\ $\dim_{\mathbb{C}}E = n-2$ \end{tabular} \\ \hline
\end{tabular}
\end{center}
\medskip

\item The following table lists all the possibilities for the CR structure
of $S$ near $p$  when the codimension is $d=3$ ($\dagger$ marks the case when it
is not known if $S$ must necessarily be Levi-flat).
\medskip

\begin{center}
\begin{tabular}{|c|c|c|c|}
\hline
\diaghead{dimRSCRdimS}{$\dim_{\R}(S)$}{$\CRdim(S)$}
 & $n-4$ & $n-3$ & $n-2$ \\
\hline
$2n-6$ & DNO & \begin{tabular}{c}  $S$ is complex \\ $\dim_{\mathbb{C}}E = n-3$ \end{tabular}& DNO \\ \hline
$2n-5$ & \begin{tabular}{c} $S$ is Levi-flat \\$\dim_{\mathbb{C}}E= n-4$ \end{tabular}& \begin{tabular}{c} $S$ is Levi-flat \\
$\dim_{\mathbb{C}}E = n-3$ \\ \hline $\dim_{\mathbb{C}}E = n-4\ \dagger$ \end{tabular} & DNO \\ \hline
$2n-4$ & DNO & \begin{tabular}{c} $S$ is Levi-flat \\ $\dim_{\mathbb{C}}E = n-3$ \end{tabular} & \begin{tabular}{c} $S$ is complex\\
$\dim_{\mathbb{C}}E = n-3$\end{tabular}  \\ \hline
\end{tabular}
\end{center}
\end{enumerate}
\end{cor}
\medskip

We see that when codimension $d=2$, then $S$ must be Levi-flat or complex,
and in fact we understand precisely the CR structure of $S$ at a generic
point.  Also given the examples in Section~\ref{section:examples}, all the
possibilities for $d=2$ actually occur.

When $d=3$ there is one case where we cannot decide if $S$ is Levi-flat or
not using Theorem~\ref{thm:generalsingthm}.

\begin{proof}
Note that as $p$ is generic, we can assume that $S$ is
a CR submanifold near $p$ and $p$ can be chosen such that
$\dim_\C S \cap L_q$ is constant on $q \in S$. Here, $L_q$ is the leaf through
$q$.

Let us start with $d=2$.
Let us first note that
the dimension of $S$ must be greater than
or equal to $2(n-d) = 2n-4$, so the possibilities are
$2n-4$ and $2n-3$.  Since $S \cap L$ is contained in $S$,
the CR dimension of $S$ must be greater than or equal to that of $S \cap L$. Hence
it is either $n-3$ or $n-2$.  If $\CRdim S = n-3$, then
$\dim_\C S \cap L = n-3$ of course and so $S$ must
be Levi-flat.  By
Theorem~\ref{thm:generalsingthm}, $S$ intersects all leaves near $L$ and
hence must have real
dimension $2(n-3)+2= 2n-4$ as there is a 2-dimensional family of
leaves and the intersection with each of them
is of real dimension $2n-6$.

So now consider $\CRdim(S) = n-2$.  If $\dim S = 2n-4$, then $S$ must be
complex.  Both $\dim_\C S \cap L = n-3$ and $n-2$ are possible.
When $\dim S = 2n-3$, then necessarily $\dim_\C S \cap L = n-2$ by the above
argument, and so $S$ is Levi-flat.

The case $d=3$ follows similarly.  In this case we have
$2(n-d) = 2n-6 \leq \dim S \leq 2n-4$, and $\CRdim S \geq n-4$ as
$\dim_\C S \cap L \geq n-4$.  Similarly, when $\dim_\C S\cap L = n-4$,
then $\dim S = 2(n-4)+3 = 2n-5$.  The table fills in similarly as above.

The only case when we do not know if $S$ is flat or not is when $\dim_\C S \cap L = n-4$
and the CR dimension of $S$ is $n-3$.
\end{proof}

\begin{proof}[Proof of Theorem~\ref{thm:lfsing}]
Let us suppose that $f(0) = 0$ and $0 \in S$ for simplicity in the following
arguments.

Let us begin with \eqref{lfsing:item1}.  Suppose we have a real-analytic CR
map $f(x,\xi)$ from $U$ onto $M$, where $x \in \R^d$ and $\xi \in
\C^{n-d}$.  It follows that $f$ is holomorphic in $\xi$.
As in the proof of Lemma~\ref{lemma:singjacob}, the map $f$ extends to a holomorphic
maps $F(z,\xi)$ defined in an open set $\hat{U} \subset \C^d\times \C^{n-d}$
by simply replacing the real variable $x$ by a complex variable $z$.
$F$ sends $U$, as a generic submanifold of $\C^n$ defined by $z=\bar{z}$ in $\hat{U}$,
diffeomorphically onto $M$.  By Lemma~\ref{lemma:singjacob},
\begin{equation}
f^{-1}(S) = \{(z,\xi)\in \hat{U}: J_F(z,\xi) = 0\} \cap \{z=\bar{z}\}.
\end{equation}
Thus, in $(x,\xi)$-coordinates, $f^{-1}(S)$ is given by the vanishing of the
function $J_F(x,\xi)$ on $U$,
where $J_F(x,\xi)$ is real-analytic in $x$ and holomorphic in $\xi$.

For the proof of \eqref{lfsing:item2}, let $L$ be a leaf on $M$. Since the leaves of $M$ is parametrized
by $f(x,\xi)$, where $x$ is regarded as
a parameter, $L$ is a complex manifold of dimension $n-d$.
The conclusion of \eqref{lfsing:item2}
then follows from Theorem~\ref{thm:generalsingthm}.

Let us now prove \eqref{lfsing:item3}.  Fix $p \in S$ and
take a leaf $L$ of $M$ through $p$, and suppose $f(0) = p$.
Suppose that $S \cap L$ is $(n-d-1)$-dimensional.  Take
a one-dimensional curve $x \colon [0,\epsilon) \to \R^d$ such that
the leaves of $M$ given by $L_t = f(\{x(t)\} \times \C^{n-d} \cap U)$ do
not intersect $S$ for all $t > 0$.  We find the family of disks
$A \colon [0,\epsilon) \to \overline{\Delta}$ such that
$A(0,\overline{\Delta}) \setminus S \not= \emptyset$ (as $S \cap L$
is $n-d-1$ dimensional), and $A(t,\overline{\Delta}) \subset L_t$.
We can apply Theorem~\ref{thm:generalsingthm} to show that $S \cap L_t$
is nonempty.  As the curve $x$ was arbitrary, we are done.
\end{proof}

Let us prove a general proposition about identifying the CR singular set
for codimension two submanifolds.  It is particularly useful for computing
examples.

\begin{prop} \label{prop:singcomp}
Let $w = \rho(z,\bar{z})$ define a CR singular manifold $M$ of in coordinates
$Z=(z,w) \in \C^{n-1} \times \C$,
where $\rho$ is real-analytic such that $\rho = 0$ and $d\rho = 0$ at the origin.
Then, the CR singularity $S$ is defined precisely by
\begin{equation}\label{eqdefS}
S = \{ (z,w) \in \C^n :
\rho_{\bar{z}_k} (z,\bar{z}) = 0\
\text{for}\
k=1,\dots , n-1\} .
\end{equation}
\end{prop}

\begin{proof}[Sketch of Proof]
We can take $r_1 = \Re(w-\rho)$ and $r_2 = \Im(w- \rho)$ to be real-valued
defining equation for $M$. As explained in Section~\ref{section:intro},
the CR singular set $S$ consists of those points in $M$ where the matrix (with $Z=(z,w)$)
\begin{equation}
\left[\frac{\partial r_j}{\partial \bar{Z}_k}\right]_{j,k}= \begin{bmatrix}
  \frac{-\rho_{\bar z_1}-\bar\rho_{\bar z_1}}{2} & \frac{-\rho_{\bar z_2}-\bar\rho_{\bar z_2}}{2}&\cdots & \frac{-\rho_{\bar z_{n-1}}-\bar\rho_{\bar z_{n-1}}}{2} & \frac{1}{2}\\
  \frac{-\rho_{\bar z_1}+\bar\rho_{\bar z_1}}{2i} & \frac{-\rho_{\bar z_2}+\bar\rho_{\bar z_2}}{2i}&\cdots & \frac{-\rho_{\bar z_{n-1}}+\bar\rho_{\bar z_{n-1}}}{2i} & -\frac{1}{2i}
\end{bmatrix}
\end{equation}
is of rank at most one, which  can only happen when $\rho_{\bar{z}_k}(z,\bar{z}) = 0, k=1,\dots, n-1$. Therefore, 
 \eqref{eqdefS} follows.
\end{proof}
We conclude this section by the following example of a
Levi-flat codimension two submanifold $M$ of $
\C^{n-1}_z \times \C_w$ ($n\geq 3$) whose CR singular set is isolated.  Such a manifold is
then not an diffeomorphic image of
a codimension two generic submanifold on $\bC^{n}$ under CR map.
This conclusion follows by Theorem~\ref{thm:lfsing}.
\begin{example}\label{ex:notanimage}
Let $M$ be given by
\begin{equation}
w = \Re (z_1^2 + z_2^2 + \dots + z_{n-1}^2 ).
\end{equation}
The CR singular set of $M$ is the origin, by
Proposition~\ref{prop:singcomp}. Furthermore, $M\setminus \{0\}$
is Levi-flat. Assume that there exist a generic codimension two submanifold $N\subset \bC^{n}$
and an analytic CR map $f\colon N \to M$ that is diffeomorphism
onto $M$. Then $N\setminus f^{-1}(S)$ is Levi-flat and so is $N$.
From Theorem~\ref{thm:lfsing} we obtain that $S$ is of dimension at least $n-2\geq 1$.
This is a contradiction.

The manifold $M$ is the intersection of two nonsingular Levi-flat hypersurfaces,
One defined by $\Im w = 0$, and
the other by
$\Re w = \Re (z_1^2 + \dots + z_{n-1}^2)$.  Further,
the manifold $M$ contains the singular complex analytic set
$\{ z_1^2 + z_2^2 + \dots + z_{n-1}^2 = 0 , w = 0 \}$ through the origin.
\end{example}


\section{Singular coordinates for nowhere minimal finite images}
\label{section:singcoord}

We would like to find at least a partial analogue for CR singular manifolds
of the following standard result for CR manifolds.  If $M$ is a CR
submanifold, and $p \in M$ then the CR orbit $\Orb_p$ is the germ of the
smallest CR submanifold of $M$ of the same CR dimension as $M$
through $p$.  For a real-analytic $M$, the CR orbit exists and is unique by
a theorem of Nagano (see \cite{BER:book}).  Near a generic point where the
orbit is of maximal possible dimension in $M$, the CR orbits give a
real-analytic foliation of $M$.  The following theorem gives a way
to describe this foliation.

\begin{thm}[see \cite{BER:craut}]
\label{nicenormcoord}
Let $M \subset \C^n$ be a
generic
real-analytic nowhere minimal
submanifold of real codimension $d$, and let $p \in M$.
Suppose that all the CR orbits are of real codimension $j$
in $M$.
Then there are local holomorphic coordinates
$(z,w^\prime,
w^{\prime\prime})\in\C^k\times\C^{d-j}\times\C^j=\C^n$,
vanishing at $p$, such that near $p$, $M$ is
defined by
\begin{align}
&\Im w^\prime=\varphi(z,\bar z,\Re
w^\prime,\Re w^{\prime\prime}) ,
\\&\Im w^{\prime\prime}=0,
\end{align}
where $\varphi$ is a real valued real-analytic function with
$\varphi(z,0,s^\prime, s^{\prime\prime})\equiv 0$.
Moreover,
the local CR orbit of the point
$(z,w^\prime,w^{\prime\prime})=(0,0,s'')$, for
$s'' \in \R^j$, is given by
\begin{align}
&\Im w^\prime=\varphi(z,\bar z,\Re w^\prime,s'') ,
\\&w^{\prime\prime}=s'' .
\end{align}
\end{thm}

For convenience, we will call a subvariety of codimension one a
\emph{hypervariety}.
For a real hypervariety $H \subset \C^n$ let $H^*$ denote the set of points near
which $H$ is a real-analytic nonsingular hypersurface.  We say
$H$ is a Levi-flat hypervariety if $H^*$ is Levi-flat.
The subvariety defined by
\begin{equation}
\Im w''_j = 0
\end{equation}
is a Levi-flat hypervariety (in this case, it is nonsingular).
We cannot
find coordinates as in Theorem~\ref{nicenormcoord}
for a CR singular manifold, but we can at least find
Levi-flat hypervarieties that play the role of
$\{ \Im w''_j = 0 \}$.  We should note that \emph{not} every Levi-flat hypervariety
is of the form $\Im h = 0$ for some holomorphic function $h$ (see
\cite{BG:lf}).

\begin{thm}
Let $N \subset \C^n$ be a real-analytic generic  connected submanifold
and
let $f \colon N \to \C^n$ be a real-analytic CR map
that is
a diffeomorphism onto its image, $M = f(N)$.
Suppose that $f$ extends to a finite holomorphic
map $F$ from a neighborhood of $N$ to a neighborhood of $M$.
Suppose that
all the CR orbits of $N$ are of real codimension $j$ in $N$,
and $p \in M$ is such that $M$ is CR singular at $p$.
Then there exists a neighborhood $U$ of $p$
and
$j$ distinct Levi-flat hypervarieties $H_1, H_2, \dots, H_j$
such that
$\dim_\R H_1 \cap \dots \cap H_j = 2n-j$ and
\begin{equation}
M \subset H_1 \cap \dots \cap H_j .
\end{equation}
Furthermore, if $\Orb_q$ is a germ of a CR orbit of $N$ at $q \in N$,
Then there exists an $n-j$ dimensional germ of a complex variety
$\bigl(L,f(q)\bigr)$ with $\bigl(L,f(q)\bigr) \subset \bigl(H_k,f(q)\bigr)$
for all $k=1,\dots,j$ and as germs
\begin{equation}
f(\Orb_q) \subset \bigl(L,f(q)\bigr) .
\end{equation}
\end{thm}

In particular, if $N$ is Levi-flat then we can find Levi-flat hypervarieties
$H_1,\dots,H_j$ such that $M$ is one of the components of
$H_1 \cap \dots \cap H_j$.

\begin{proof}
Suppose that $q \in N$ is such that $F(q) = p$.
We will from now on assume $F \colon V \to U$ is a representative of the germ
for some connected open subsets
$V, U \subset \C^n$, $q \in V$,
$p=F(q) \in U$, and $F(V) = U$.  As the germ of $F$ at $q$ is finite we
will assume that the representative $F$ is proper map from $V$ onto $U$.
Let us also assume that $N$
is closed in $V$ and furthermore that
the $N$ is defined in $V$ using coordinates of
Theorem~\ref{nicenormcoord} (the coordinates are defined in all of $V$).

Fix a nonzero vector $v \in \R^j$.
Take the variety $\{ \Im \langle w'', v \rangle = 0 \}$ and let us
push it forward by $F$.
The image need not necessarily be a real-analytic subvariety.
We claim, however, that as $F$ is finite, then
$F(\{ \Im \langle w'', v \rangle = 0 \})$ is contained in a real-analytic
subvariety of codimension one.  To show the claim we complexify $F$ and
$\{ \Im \langle w'', v \rangle = 0 \}$, push the set forward using
the Remmert proper map theorem, and then restrict back to the diagonal.

As $F$ is proper, then
the function $\sF(\zeta,\xi) = \bigl(F(\zeta),\bar{F}(\xi)\bigr)$ is
a proper map of $V \times V^*$ to $U \times U^*$
where $V^* = \{ \xi : \bar{\xi} \in V \}$.  So as
$\{ \Im \langle w'', v \rangle = 0 \}$ complexifies to a complex submanifold $\sH \subset
V \times V^*$, then as $\sF$ is proper, $\sF(\sH)$ is an irreducible
complex subvariety
of $U \times U^*$.
Let $(\zeta',\xi')$ denote the coordinates in $U \times U^*$ and
$\pi_{\zeta'}$ the projection onto the $\zeta'$ coordinates.
Let $H$ denote the set $\pi_{\zeta'} \bigl( \sF(\sH) \cap
\{ \overline{\zeta'} = \xi' \} \bigr)$.
The defining equation for $\sF(\sH)$ defines $H$ once we plug in
$\overline{\zeta'}$ for $\xi'$.
Therefore $H$ is an irreducible real subvariety of $U$ and $M \subset H$.
A holomorphic function that is not identically zero cannot vanish
identically on the maximally totally real set
$\{ \overline{\zeta'} = \xi' \}$, and hence $H$ must be a proper subvariety
of $U$.
By construction,
$F(\{ \Im \langle w'', v \rangle = 0 \}) \subset H$.  Therefore,
the subvariety $H$ must be of real codimension one as $F$ is finite.

Since
$\{ \Im \langle w'', v \rangle = 0 \}$ is a real Levi-flat
hypersurface and $F$ is a local
biholomorphism outside of a complex subvariety, we see that $H$ must
be Levi-flat at some point.  By a lemma of Burns and Gong (see \cite{BG:lf} or
\cite{Lebl:lfnm}) then as $H$ is irreducible, it is Levi-flat at all
smooth points of top dimension, and hence Levi-flat by definition.

Suppose we have taken $k$ linearly independent vectors $v_1,\dots,v_k$ such that
the corresponding $H_1,\dots,H_k$ have an intersection that is of real
codimension $k$.  Suppose that $k < j$. 
We have
\begin{equation}
H_1 \cap \dots \cap H_k =
\pi_{\zeta'} \bigl( \sF(\sH_1) \cap \dots \cap \sF(\sH_k)
\cap \{ \overline{\zeta'} = \xi' \} \bigr) .
\end{equation}
Let
\begin{equation}
\sV =
\sF^{-1} \bigl( \sF(\sH_1) \cap \dots \cap \sF(\sH_k) \bigr) .
\end{equation}
The variety $\sV$ has codimension $k$.
Let us treat
$(z,w',w'')$ and $(\bar{z},\bar{w}',\bar{w}'')$ as different variables.

If $\sV = \sH_1 \cap \dots \cap \sH_k$, then pick any vector
vector $v \in \R^j$ linearly independent from $v_1,\dots,v_k$,
and let $\sH$ be defined by
$\{ \langle w'',v\rangle - \overline{\langle w'',v\rangle} = 0\}$.
Then the intersection $\sV \cap \sF^{-1}\bigl(\sF(\sH)\bigr)$ is of
codimension $k+1$.
It now follows that
$\sF(\sH_1) \cap \dots \cap \sF(\sH_k) \cap \sF(\sH)$ is of codimension $k+1$.
And hence if $H = \pi_{\zeta'} \bigl( \sF(\sH) \cap
\{ \overline{\zeta'} = \xi' \} \bigr)$, then
$H_1 \cap \dots \cap H_k \cap H$ has real codimension $k+1$.  This claim
follows because if $H_1 \cap \dots \cap H_k \cap H$ has codimension $k$,
it would have some point where it is a smooth real codimension $k$ manifold.
The complexification at that point would have to be a complex codimension $k$
manifold, and we know that is not true.

In case $\sV$ has other components, then
let $\sC$ be any irreducible component of $\sV$ that is not contained
in $\sH_1 \cap \dots \cap \sH_k$.
Let us treat
$(z,w',w'')$ and $(\bar{z},\bar{w}',\bar{w}'')$ as different variables
as usual.  Note that $\sX = \sF^{-1}\bigl( \sF(\{ w'' = \bar{w}''\}) \bigr)$
is of dimension $2n-j$ as $\sF$ is finite.  Furthermore $\sX \subset \sV$.
There must exist a point
$x = (z_0,w_0',w_0'',\bar{z}_0,\bar{w}_0',\bar{w}_0'') \in \sC$ where
$x \notin \sX$.
Hence $\sF^{-1}\bigl( \sF(x) \bigr) \notin \sX$.
In particular, $\sF^{-1}\bigl( \sF(x) \bigr) \cap \{ w'' = \bar{w}'' \}$
is the empty set.
We can pick a vector $v \in \R^j$
linearly independent from $v_1,\dots,v_k$,
such that the set $\sH$ defined by
$\{ \langle w'',v\rangle - \overline{\langle w'',v\rangle} = 0\}$ does not
contain any point of $\sF^{-1}\bigl( \sF(x) \bigr)$
and so the intersection $\sV \cap \sF^{-1}\bigl(\sF(\sH)\bigr)$ is of
codimension $k+1$.  We then proceed as above.

Hence we can find
$j$ distinct Levi-flat hypervarieties $H_1, H_2, \dots, H_j$
such that
$\dim_\R (H_1 \cap \dots \cap H_j) = 2n-j$ and
$M \subset H_1 \cap \dots \cap H_j$.

To find $L$, we push forward the
complex variety $\{ w'' = s'' \}$ by $F$, which is finite.
\end{proof}


\begin{remark}
Theorem~\ref{nicenormcoord} also generalizes to some extent to certain CR
manifolds at points where the dimension of the CR orbits is not constant and
hence where the CR orbits do not form a foliation.  At such singular points
it is not always true that such $N$ lie inside Levi-flat hypervarieties,
despite all CR orbits being of positive codimension.  See \cite{Lebl:lfnm}
for more on these matters.
\end{remark}


\section{Failure of extensions of real-analytic CR functions} \label{section:noext}
In this section, we focus our attention on functions satisfying the
pointwise Cauchy-Riemann conditions on a CR singular
image with a nonempty CR singular set $S$.
We shall show for each
$p\in S$, there exists a real-analytic function on a neighborhood of $p$
in $M$ satisfying all the pointwise Cauchy-Riemann conditions that does not
extend to a holomorphic function at $p$.  This result
generalizes Lemma~\ref{lemma:nonextendible}.

\begin{thm}
Let $M\subset \bC^n$ be a connected real-analytic CR singular submanifold such that there are a
real-analytic generic submanifold $N \subset \C^n$ and a real-analytic CR map $f \colon N \to \C^n$
that is a diffeomorphism onto $M = f(N)$. Suppose $S$ is the CR singular set of $M$ and $M\setminus S$
is generic.
Then, for any $p\in S$, there exist a neighborhood $U$ of $p$ and a real-analytic  function $u$ on $U\cap M$ such that 
$Lu|_q =0$ for any $q\in U\cap M$ and any $L\in T^{(0,1)}_qM$,  but $u$ does not
extend to a holomorphic function on any neighborhood of $p$
in $\C^n$.
\end{thm}

\begin{proof}
Let $F$ be the unique holomorphic extension of $f$ to a neighborhood $V$ of
$f^{-1}(p)$ in $\bC^n$. Let $\theta = J_F$ and $\varphi= \theta^2 \circ
f^{-1}$. Then $\varphi$ is a real-analytic function on $f(V\cap N)$. We
claim that $\varphi$ satisfies all CR conditions on $M$. Indeed, for any point $q\in M$ near $p$, let $L\in T^{(0,1)}_qM$. If $q\in S$, then $q':=f^{-1}(q) \in f^{-1}(S) \subset \{\theta=0\}$ and hence $\theta(q')=0$. If $q\in M\setminus S$, then $X_{q'}:=(f^{-1})_*L\in T^{(0,1)}_{q'}N$
 and thus $X_{q'}(\theta)(q') = 0$. Therefore for all $q \in U\cap M$, we have
\begin{equation}
(L\varphi)(q) = L(\theta^2\circ f^{-1})(q) =
((f^{-1})_*L)\theta^2\bigl(f^{-1}(q)\bigr) =
X_{q'}\theta^2(q')=
2\theta(q')\,\bigl(X_{q'}\theta\bigr)(q') = 0.
\end{equation}
Therefore, the claim follows.

If $\varphi$ does not extend to a holomorphic function in a neighborhood of $p$ in $\bC^n$, then we are done.
Otherwise, suppose that $\varphi$ extends to a neighborhood. Notice that $\varphi \equiv 0$ on $S$. Without loss of generality we can assume further that $\varphi$ is radical.

On the other hand, by Lemma~\ref{lemma:nonextendible}, we can find a real-analytic function $u$
on a neighborhood $U$ of $p$ in $M$ such that $u$ is CR on $U\setminus S$ and $u$ does not extend holomorphically
near $p$.  By construction,  $u\varphi$ and $u^2\varphi$ restricted to $U$ are both CR functions on $M\cap U$. We claim that at least one of the two functions $u\varphi$ and
$u^2\varphi$ does not extend holomorphically past $p$. Indeed, assume for a contradiction that $v_1$
and $v_2$ are holomorphic functions on a neighborhood of $p$ in $\C^n$ whose
restrictions to $M$ are $u\varphi$ and $u^2\varphi$, respectively. Observe
that the following equalities hold on $U$.
\begin{equation}\label{vovanqua}
v_1^2 = u^2\varphi^2 = v_2\varphi.
\end{equation}
Since $M$ is generic at all points on $M\setminus S$, $v_1^2$ and $v_2\varphi$ are
holomorphic, we deduce from \eqref{vovanqua} that $v_1^2 = v_2\varphi$ in a neighborhood of $p$ in $\bC^n$.
In other words, we have the following equality in the ring $\sO_p$.
\begin{equation}\label{hanh}
 v_1^2 = v_2\varphi.
\end{equation}
Note that the ring $\sO_p$ is a unique factorization domain and $\varphi$ is
radical. From \eqref{hanh} we obtain that $v_1$ divides $v_2$ and hence $\frac{v_2}{v_1}$ is holomorphic near $p$.
Consequently, $u$ extends to the holomorphic function $\frac{v_2}{v_1}$ on a neighborhood of $p$.
We obtain a contradiction. 
\end{proof}


\section{Examples} \label{section:examples}

We start this section by the following proposition, which is helpful in
constructing examples.

\begin{prop} \label{prop:lf3simp}
Let $w = \rho(z,\bar{z})$ define a connected CR singular manifold $M$ near the origin
in coordinates $(z,w) \in \C^2 \times \C$, where $\rho$ is real-analytic and
such that $\rho=0$ and $d\rho = 0$ at the origin.
If $\rho_{\bar{z}_1} \equiv 0$, then
$M$ is Levi-flat at CR points,
and furthermore,
the set $S$ of CR singularities is given by $M \cap \{ (z,w) : \rho_{\bar{z}_2}(z) = 0 \}$.

Furthermore, for each point $p \in M$, there exists a neighborhood $U$ such
that $U \cap M$ is the image under a real-analytic CR diffeomorphism of an
open subset of $\R^2 \times \C$.
\end{prop}

Note that if $w = \rho(z,\bar{z})$ and $\rho_{\bar{z}_1} \equiv 0$,
we could also get that $M$ is a complex manifold,
but in this case $M$ is not CR singular.

\begin{proof}
If $\rho_{\bar z_2} \equiv 0$ then $\rho$ is holomorphic and hence $M$ is
complex analytic. Otherwise, $\rho_{\bar z_2} \not\equiv 0$ and therefore,
from Proposition~\ref{prop:singcomp}, we see that
\begin{equation}
S = M\cap \{(z,w): \rho_{\bar z_2}(z,\bar z) = 0\}.
\end{equation}
Hence, $S\subset M$ is a proper real subvarieties and $M\setminus S$ is generic.
To see that $M\setminus S$ is Levi-flat, observe that in a neighborhood of $p\not\in S$, $M\setminus S$ is foliated by family of one-dimensional complex submanifolds defined by $L_{t} = \{(z,w): w = \rho(z_1,t,0, \bar t)\}$ with complex parameter $t$.
Let a local map $f \colon \R^2\times \C \to M$ be given by
\begin{equation}
f \colon (x,y,\xi) \mapsto \bigl(\xi, x+iy,\rho(\xi, x+iy, 0 , x-iy)\bigr).
\end{equation}
The CR structure on $\R^2\times \C$ is given by $\partial
/\partial \bar \xi$ and so clearly $f$ is CR map. Since $\rho$ does not
depend on $\bar z_1$, it follows that $f$ sends $\R^2\times \C$ into $M$.
The fact that $f$ is local diffeomorphism is immediate.
\end{proof}

Using the proposition we can easily create many examples showing that
the CR singular set of a Levi-flat manifold that is an image of a
CR diffeomorphisms can have any possible CR structure allowed
by Corollary~\ref{cor:lfcor}.

\begin{example}
We can obtain a 3-dimensional CR singularity by simply taking a parabolic
CR singular Bishop surface in 2 dimensions and considering it in 3
dimensions.  For example,
\begin{equation}
w = \abs{z_2}^2 + \frac{\bar{z}_2^2}{2} .
\end{equation}
The manifold is the image of $\R^2 \times \C$ by the construction of
Proposition~\ref{prop:lf3simp}.
The CR singular set is the set $\{ \Re z_2 = 0 \} \cap M$, hence 3
real dimensional.

The submanifold is contained in the nonsingular Levi-flat hypersurface
defined by
$\Im w = - \Im \frac{z_2^2}{2}$.
%
\end{example}

\begin{example}
Next, let us consider
\begin{equation}
w = z_1 \bar{z}_2^2 .
\end{equation}
The manifold is the image of $\R^2 \times \C$ by the construction of
Proposition~\ref{prop:lf3simp}.  The CR singular
set is the set $( \{ z_1 = 0 \} \cup \{ z_2 = 0 \} ) \cap M$,
that is a union of two 2-dimensional sets, both of which are
complex analytic.  Note that the set $\{ z_1 = 0 \} \cap M$ is
a complex analytic set that is an image of a totally real
submanifold of $\R^2 \times \C$ under the map of
Proposition~\ref{prop:lf3simp}.  We therefore have a complex analytic
set that is a subset of $M$ while not being an image of
one of the leaves of the Levi-foliation of $\R^2 \times \C$.
%
%
\end{example}

\begin{example}
Consider
\begin{equation}
w = z_1 \bar{z}_2 - \frac{\bar{z}_2^2}{2} .
\end{equation}
Again the manifold is the image of $\R^2 \times \C$.  The CR singular
set $S$ is the set $\{ z_1 = \bar{z}_2 \} \cap M$,
which is a totally real set; to see this fact
simply substitute $\bar{z}_2 = z_1$ in the defining equation for $M$ to
find that $S$ is the intersection of
$\{ z_1 = \bar{z}_2 \}$ with a complex manifold.
\end{example}

\begin{example}
Consider
\begin{equation}
w = z_1 \bar{z}_2 - \frac{z_2 \bar{z}_2^2}{2} .
\end{equation}
The CR singular set $S$ is the set $\{ z_1 = \abs{z_2}^2 \} \cap M$,
which is a CR singular submanifold.
\end{example}

\begin{example} \label{example:ER}
While we have mostly concerned ourselves with flat manifolds, there is
nothing particularly special about flat manifolds.  Even a finite-type
manifold can map to a CR singular manifold.  The following
example was given in
\cite{ER}*{Example 1.6}.
Let $M \subset \C^3$ be given by
\begin{equation}
M = \left\{(z, w_1 , w_2 ) \in \C^3 : \Im w_1 = \frac{{\abs{z}}^2}{2},
\Im w_2 = \frac{{\abs{z}}^4}{2} \right\} .
\end{equation}
$M$ is taken to the CR singular manifold
\begin{equation}
\{(z_1, z_2, w) \in \C^3 : w = {(\bar{z}_2+i{\abs{z_1}}^2+{\abs{z_1}}^4 )}^2 \}
\end{equation}
via the finite holomorphic map
\begin{equation}
(z, w_1, w_2) \mapsto
\bigl(z, w_1 + iw_2 , {(w_1 - iw_2 )}^2 \bigr) .
\end{equation}
The map is a diffeomorphism onto its image when restricted to $M$.
\end{example}

%
%
%
%


\def\MR#1{\relax\ifhmode\unskip\spacefactor3000 \space\fi%
  \href{http://www.ams.org/mathscinet-getitem?mr=#1}{MR#1}}

\end{document}